\documentclass[a4paper,12pt]{article}

\usepackage[austrian,american]{babel}

\usepackage[intlimits,leqno]{amsmath}
\usepackage{amssymb} 
\usepackage{amsthm}
\usepackage{graphicx,color}
\usepackage{makeidx}
\usepackage{epsfig}
\usepackage{latexsym}
\usepackage{pstricks,pst-node,amsfonts}

\theoremstyle{plain}
\newtheorem{theo}{Theorem}
\newtheorem{prop}{Proposition}
\newtheorem{coro}{Corollary}
\newtheorem{lemm}{Lemma}
\newtheorem{rema}{Remark}
\newtheorem{defi}{Definition}

\newtheorem{exam}{Example}

\newcommand\Z{{\mathbb{Z}}}
\newcommand\N{{\mathbb{N}}}
\newcommand\R{{\mathbb{R}}}
\newcommand\C{{\mathbb{C}}}
\newcommand{\A}{\mathcal{A}}
\newcommand{\B}{\mathcal{B}}
\newcommand{\M}{\mathcal{M}}

\renewcommand{\L}{\mathcal{L}}
\newcommand{\F}{\mathcal{F}}
\newcommand{\T}{\mathcal{T}}

\newcommand{\D}{\mathcal{D}} 
\renewcommand{\i}{{\rm i}}

\newcommand{\x}{\mathbf{x}}

\renewcommand{\d}{{\rm d}}
\newcommand{\supp}{ { \rm{supp} } }

\renewcommand\T{ T }
\renewcommand\t{ \tau }

\newcommand\K{ \mathcal{K} }

\newcommand\G{ \mathcal{G} }

\newcommand\Id{ \mbox{\,Id\,} }

\renewcommand{\S}{\mathcal{S}}

\begin{document}

\title{Dynamical systems with finite stopping times. Part~1: Relaxation, oscillation and their application to diffusion 
and wave dissipation}

\author{ Richard Kowar\\
Department of Mathematics, University of Innsbruck, \\
Technikerstrasse 21a, A-6020, Innsbruck, Austria
}

\maketitle

\begin{abstract}
In this paper, we derive general theorems for controlling (vector-valued) first order ordinary differential 
equations such that its solutions stop at a finite time $\T>0$ and apply them to relaxation and dissipative 
oscillation processes. We discuss several interesting examples for relaxation processes with finite stopping 
time and their energy behaviour. Our results on relaxation and dissipative oscillations enable us to model 
diffusion processes with finite front speeds and dissipative waves that cause in each space point $x$ an 
oscillation with a finite stopping time $\T(x)$. In the latter case, we derive the relation between $\T(0)$ and 
$\T(x)$.  Moreover, the relations beteween the control functions in the ode model and the respective pde model 
are derived.In particular, we present an application of the Paley-Wiener-Schwartz Theorem that is used in our 
analysis. A complementary approach for dissipative oscillations and its application to dissipative waves 
is presented in~\cite{Ko19b}, where the finite stopping time is achieved due to nonconstant coefficients
in second order odes. 
\end{abstract}

\section{Introduction}

In this paper, we investigate two "categories" of control problems for dynamical systems and investigate their 
relations. A central problem of the first category is of the form
\begin{equation}\label{DynSysCat1}
    u'(t) + A(t)\,u(t) = - \ell(t)\,u_0 
             \quad \mbox{on $\R$} \qquad\mbox{with}\qquad 
    u(0) = u_0 \in\R \,,
\end{equation}
where $A:\R\to\R^{n\times n}$ is a continuous function and $\ell$ is a 
\emph{control function} that forces the solution $u$ to \emph{stop at a finite time} $\T>0$, i.e. $u(t)=0$ for 
$t\geq\T$. 
For this category, we derive general theorems and then focus on \emph{relaxation} and 
\emph{dissipative oscillation} processes. We give a thorough analysis of relaxation processes 
including its energy behaviour and confine our analysis of oscillation to the essentials, 
due to shortage of space. An eleborate investigation of dissipative oscillations with nonconstant coefficients 
and/or finite stopping time is presented in~\cite{Ko19b}. This category is mainly a means to an end to investigate 
the second category. This one is of the form 
\begin{equation}\label{DynSysCat2}
    \A\,u = f - \L *_{x,t} f
             \quad \mbox{on $\R^2$} \qquad\mbox{with}\qquad 
    u|_{t<0} = 0 \,,
\end{equation}
where $\A$ is the \emph{diffusion} or the \emph{dissipative wave operator} with constant coefficient and 
$\L$ is a \emph{control function}. By $*_s$ we denote the convolution with respect to 
the variable $s$. In case of diffusion, we are interested in control functions $\L$ that guarantee 
a diffusion front propagating with finite speed. More precisely, if $\G_{diff}$ solves~(\ref{DynSysCat2}) 
with $f(x,t) = \delta(x)\,\delta(t)$, then $\L$ is such that 
\begin{equation*}
     \supp( \G_{diff}(\cdot,t) ) = \{ x\in\R\,|\, |x| \leq R(t)] \quad\mbox{with}\quad  R'(t)\leq c_F 
       \quad\mbox{for}\quad t > 0 
\end{equation*}
is guaranteed, for some constant $c_F\in (0,\infty)$. Here we focus on the space time domain $\R^2$, 
but the the generalization to the domain $\R^{n\times 1}$ is readily carried out. 
In case of dissipative wave propagation, the control functions $\L$ is modelled such that the oscillation at 
a fixed point $x$ in space caused by a dissipative spherical wave stops at a finite time. More precisely, 
if $\G_{diss}$ solves~(\ref{DynSysCat2}) with $f(x,t) = \delta(x)\,\delta(t)$ (delta distributions), 
then $\L$ is such that 
\begin{equation*}
     \supp( \G_{diss}(x,\cdot) ) = [T_1(x),T_1(x) + T(x)] \quad\mbox{with}\quad  0\leq T(x) < \infty
\end{equation*}
is guaranteed, where $T_1(x)$ denotes the travel time of the wave propagating from the origin $0$ to the 
position $x\in\R^3$. In particular we analyse the structure of $\T(x)$. 
We show that relaxation processes from the first category can be used to model diffusion processes from 
the second category. And similarly, relaxations and (certain) dissipative oscillations from the first category 
can be used to model dissipative wave propagation from the second category. Moreover, the relation between 
the respective control functions $\ell$ and $\L$ are derived.

Actually this paper contains two approaches for the second category. In one, a diffusion or dissiaptive wave, 
say $u$, is modelled with the help of a relaxation or a dissipative oscillation. This gives us "in some sense" 
an explicite representation formula. And in the other one, the process $u$ is modelled as the solution of a 
standard partial differential equation with an additional control term. This later approach does not require 
to change the "governing" operator $\A$. Of course, it is also reasonable to ask for an operator $\B$ for which 
\begin{equation}\label{DynSysCat3}
    \B\,u = f 
             \quad \mbox{on $\R^2$} \qquad\mbox{with}\qquad 
    u|_{t<0} = 0 \,
\end{equation}
is equivalent to~(\ref{DynSysCat2}). But this is beyond the scope of this paper. From our experience, it is 
very challenging to start with the third approach and with our approaches we get easily uncountable many 
examples of diffusion and dissipative waves satisfying our requirements. In particular, then it is simpler to 
find the mentioned operators $\B$. It is easy to see that the operators $\A$ and $\B$ are related by 
$$
         \A  = \B - \L_* \,\B \qquad\mbox{,where}\qquad   \L_*(\cdot) := \L *_{x,t} \cdot \,\,\,.
$$

This paper is organized as follows: In Section~\ref{sec-2}, we derive general theorems 
for controlling (vector-valued) first order ordinary differential equations such that its 
solution stops at a finite time $\T>0$ and apply them to relaxation and dissipative oscillation 
processes. Several examples for relaxation processes with finite stopping time and their energy behaviour 
is systematicaly discussed in the subsequent section. These results about relaxations and dissipative 
oscillations with finite stopping times enables us to model diffusion processes with finite front speeds 
and dissipative waves that cause in each space point an oscillation with a finite stopping time. Moreover, in 
Chapter~4 and Chapter~5 we derive the relation beteween the control functions $\ell$ and $\L$. 
In the appendix, we summarize some notations and Theorems about the Fourier transform and present an 
application of the Paley-Wiener-Schwartz Theorem that is used in our analysis. 
We conclude this paper with a short summary of our results in the section Conclusion.

\section{Control problems with finite stopping time}
\label{sec-2}

In this section, we derive basic theorems for controlling ode's such that the described process stops at a 
finite time $\T$. Prior to that, we shortly recall some general facts about ode's with time dependent 
coefficients and discuss the "no memory property" and its significance for the stopping time. 
In principle, the basic theorems show that any ode can be controlled in the claimed manner. However, 
the higher the order of the ode, the more complex the determination of the control term become.

\subsection{Ode's with time dependent coefficients in a nutshell}
\label{sec-odetimedep}

Let $t_0\in\R$, $n\in\N$ and $A:\R\to\R^{n\times n}$ be a \emph{continuous} function on $(t_0,\infty)$. We denote 
by $G(\cdot,t_0)$ the solution of  
\begin{equation}\label{eqG}
   G'(t,t_0) + A(t)\,G(t,t_0) = 0 \quad\mbox{on}\quad (t_0,\infty) \quad\mbox{with}\quad 
   G(t_0,t_0) = \Id\, .
\end{equation}
Frequently, we have $t_0=0$ and write $G(t)$ for $G(t,0)$. 
It is well known (cf. Chapter VIII in~\cite{He91}) that this problem has a unique solution and that its 
inverse exists. Moreover, we have the following theorem (cf. Chapter VIII in~\cite{He91}).

\begin{theo}\label{theo:known}
Let $t_0\in \R$, $n\in\N$, $u_0 \in \R^n$, $A:\R\to\R^{n\times n}$ and $f:\R\to\R^n$ be continuous on 
$(t_0,\infty)$  
and $G$ satisfy~(\ref{eqG}). Then 
\begin{equation}\label{equ}
    u'(t) + A(t)\,u(t) = f(t) \quad \mbox{on}\quad (t_0,\infty) \quad\mbox{with}\quad 
    u(t_0) = u_0 
\end{equation}
has a unique solution that is $C^1$ on $(t_0,\infty)$ and which reads as follows 
\begin{equation}\label{formulau}
   u(t) 
       = G(t,t_0)\,\left[ \Id\,u_0 + \int_{t_0}^t G^{-1}(s,t_0)\,f(s) \,\d s \right] 
        \quad\mbox{for}\quad
      t\in (t_0,\infty) \,.
\end{equation}
\end{theo}

\begin{proof} 
For completeness, we sketch the proof. 
It is clear that $u$ defined by~(\ref{formulau}) is continuous and satisfies $u(t_0) = G(t_0,t_0)\,u_0 = u_0$, 
due to $G(t_0,t_0) = \Id$. By differentiation of~(\ref{formulau}) and employing~(\ref{eqG}), we obtain for 
$t\in (t_0,\infty)$: 
\begin{equation*}
\begin{aligned}
  u'(t)
     &= - \left\{ A(t)\,G(t,t_0)\,\left[ \Id\,u_0 + \int_{t_0}^t G^{-1}(s,t_0)\,f(s)\,\d s \right] + f(t) \right\} \\
     &= - A(t)\,u(t) + f\,.
\end{aligned}
\end{equation*}
In particular, we see that $u$ is continuously differentiable on $(t_0,\infty)$. From this we infer that $u$ solves 
the claimed first order problem and is continuously differentiable. The uniqueness of the solution $u$ follows from 
the fact that~(\ref{eqG}) has a unique solution.  
\end{proof}

Now we show that \emph{no memory} is involved in ode's with time dependent coefficients and discuss the 
consequence for a finite \emph{stopping time} below. 

\begin{coro}\label{coro:evolprop}
Let $t_1>t_0>-\infty$, $u$ be the solution of~(\ref{equ}) and $v$ be the solution of~(\ref{equ}) with $t_0$ replaced by 
$t_1$ and $u_0 := u(t_1)$. Then $u(t) = v(t)$ for $t \geq t_1$. 
\end{coro}

\begin{proof}
Let $t_0$, $A$, $G$, $f$ and $u$ be as in Theorem~\ref{theo:known}. First we show that 
\begin{equation}\label{helpG}
    G(t,t_1)\,G(t_1,t_0) = G(t,t_0)  \qquad\mbox{for}\qquad  -\infty < t_0 < t_1 < t < \infty \,.
\end{equation}
From $F(t) := G(t,t_0)\,G^{-1}(t_1,t_0)$ and~(\ref{eqG}), we infer 
$$
    F'(t) = G'(t,t_0)\,G^{-1}(t_1,t_0) = - A(t)\, F(t)    \quad\mbox{with}\quad  F(t_1)=\Id\,,
$$
i.e. $F(t) = G(t,t_1) $ for $t_1 < t$. But this is equivalent to~(\ref{helpG}).  \\
According to Theorem~\ref{theo:known} and~(\ref{helpG}), we have for $t_0<t_1<t$: 
\begin{equation*}
\begin{aligned}
  u(t) 
    &= G(t,t_0)\,u_0 + \int_{t_0}^t G(t,s)\,f(s)\,d s \\
    &= G(t,t_1)\,G(t_1,t_0)\,u_0 + G(t,t_1)\,\int_{t_0}^{t_1} G(t_1,s)\,f(s)\,d s + \int_{t_1}^t G(t,s)\,f(s)\,d s \\
    &= G(t,t_1)\,u(t_1) + \int_{t_1}^t G(t,s)\,f(s)\,d s = v(t)\,,
\end{aligned}
\end{equation*}
as was to be shown. 
\end{proof}

Let $v$ and $u$ be as in Corollary~\ref{coro:evolprop} and let us assume that the process $v$ stops at the 
time $\T := t_1$, i.e. 
\begin{equation}\label{StopCrit1}
     v(t) = 0  \qquad\mbox{for}\qquad t>\T \,.
\end{equation}
This is our original definition of the \emph{stopping time}. Then we have $v(\T) = v'(\T) = 0$, due to the continuous 
differentiability of $v$ on $(0,\infty)$. Conversely, if $v(\T) = v'(\T) = 0$ holds then $u=0$ follows for $t\geq t_1$, 
due to the no memory property (and the uniqueness of the solution of the ode). And thus~(\ref{StopCrit1}) holds, too. 
Therefore the stopping criterium~(\ref{StopCrit1}) is equivalent to $v(\T) = v'(\T) = 0$. 
If we consider an ode of order $n$ for $v$, then the stopping time is (uniquely) specified by 
\begin{equation}\label{StopCrit2}
        v(\T) = \cdots = v^{(n-1)}(\T) = 0\,.
\end{equation}
If one considers fractional equations or other equations with memory with a finite stopping time, then the previous condition~(\ref{StopCrit2}) is not equivalent to the stopping criterium~(\ref{StopCrit1}).

We note that there are creeping processes for which 
\begin{equation*}
     v(t) = v_1  \qquad\mbox{and}\qquad v'(t) = 0 \qquad\mbox{for}\qquad t>\T \,,
\end{equation*}
is true for some constant $v_1$. If $\T$ is the smallest positive number with this property, then one may call 
$\T$ the stopping time or $v_1$ the final value of the creeping process (cf. Remark~3 and Example~7 in~\cite{Ko19b}).

\subsection{General theorems}

Without loss of generality we focus on the case $t_0=0$. For convenience, we shortly write
$G(t)$ for $G(t,0)$.

\begin{prop}\label{prop:ellvector}
Let $\T\in (0,\infty]$, $A$, $u_0$ and $G$ be as in Theorem~\ref{theo:known} (with $G(t) := G(t,0)$) and  
$\ell:\R\to\R^{n \times n}$ be continuous on $(0,\infty)$. The solution $u$ of 
$$
    u'(t) + A(t)\,u(t) = -\ell(t)\,u_0
             \quad \mbox{on $(0,\infty)$} \qquad\mbox{with}\qquad 
    u(0) = u_0
$$
satisfies $u(t)=0$ for $t\geq \T$ if and only if $\ell$ satisfies  
\begin{equation}\label{condellmatrix}
    \Id = \int_0^\T G^{-1}(s)\,\ell(s) \,\d s    \qquad\mbox{with}\qquad
    \ell(s) = 0 \quad\mbox{for}\quad s>\T \,.
\end{equation}
\end{prop}

\begin{proof}
According to Theorem~\ref{theo:known}, the solution $u$ exists, is unique and satisfies formula~(\ref{formulau}). 
Because of $G(t)\not=0$ for $t>0$,  $u(t)=0$ for $t\geq \T$ is true if and only if $\ell$ satisfies 
\begin{equation*}
   0 = \Id\,u_0 - \int_0^t G^{-1}(s)\,\ell(s)\,u_0 \,\d s 
  \quad\mbox{for}\quad
      t\geq \T \,,
\end{equation*}
which is equivalent to~(\ref{condellmatrix}). As was to be shown. 
\end{proof}

\begin{prop}\label{prop:ellvector2}
Let $\T$, $A$, $f$, $u_0$, $G$ and $\ell$ be as in Theorem~\ref{theo:known}, $\ell$ satisfy 
condition~(\ref{condellmatrix}) and $\ell^f:\R\to\R^n$ be continuous on $(0,\infty)$ and 
satisfies $\ell^f(s) = f(s)$ for $s>\T$. 
Then the solution $u$ of 
$$
    u'(t) + A(t)\,u(t) = f - \ell(t)u_0 - \ell^f(t)
             \quad \mbox{on $(0,\infty)$} \quad\mbox{with}\quad 
    u(0) = u_0
$$
satisfies $u(t)=0$ for $t\geq \T$ if and only if $\ell^f$ satisfies 
\begin{equation}\label{condellfmatrix}
    \int_0^\T G^{-1}(s)\,[f(s) - \ell^f(s)]\,\d s = 0  \,. 
\end{equation}
In particular, the unique solution $u$ is $C^1$ on $(0,\infty)$ and reads as follows 
\begin{equation*}
   u(t) 
       = G(t)\,\left[ \Id\,u_0 + \int_0^t G^{-1}(s)\,\{f(s) - \ell(s)\,u_0 - \ell^f(s)\} \,\d s \right]
  \quad\mbox{for}\quad
      t\in\R \,.
\end{equation*} 
\end{prop}

\begin{proof}
Similarly as in Proposition~\ref{prop:ellvector}, it follows that $\ell^f$ must satisfy
\begin{equation*}
   0 
       = \Id\,u_0 + \int_0^t G^{-1}(s)\,\{- \ell(s)\,u_0 + [f(s)- \ell^f(s)]\} \,\d s 
  \quad\mbox{for}\quad
      t>0 \,.
\end{equation*}
Due to $\ell(t)=0$ and $\ell^f(t) = f(t)$ for $t>\T$ and condition~(\ref{condellmatrix}), this identity 
is equivalent to~(\ref{condellfmatrix}), which proves the claim on $\ell_f$. The claimed representation formula 
for $u$ follows at once from Theorem~\ref{theo:known}, which concludes the proof. 
\end{proof}

\subsection{Control problems for relaxation}

Now we apply the above theorem to the case of relaxation.

\begin{coro}\label{coro:relaxnr1}
Let $m\in\N$, $\varrho_0>0$, $\t:[0,\infty)\to (0,\infty)$ be continuous, 
$\rho(t) = \exp\left( -\int_0^t \frac{1}{\t(s)}\,\d s \right)$ for $t\in (0,\infty)$ and $h$ be a positive, 
decreasing function on $[0,\T]$. Then 
$$
      \ell^\varrho(t) := \frac{ [h(t)-h(\T)]^m }{ \int_0^\T \rho^{-1}(s)\,[h(s)-h(\T)]^m\,\d s }\,\chi_{[0,\T)}(t) 
       \qquad\mbox{for}\qquad t\in [0,\infty)
$$
is positive and decreasing on $[0,\T]$, satisfies $(\ell^\varrho)^{(j)}(\T) = 0$ for $j = 0,\,1,\,\ldots\,m-1$ 
and the solution of the relaxation problem 
\begin{equation}\label{FullRelaxEq0}
    \t\,\varrho' + \varrho = - \t\,\varrho_0\,\ell^\varrho \quad\mbox{on}\quad (0,\infty) 
     \qquad\mbox{with}\qquad \varrho(0)=\varrho_0\,
\end{equation}
stops at the time $\T$. Moreover, it follows that $\varrho$ is $C^{m}$ on $(0,\infty)$.
\end{coro}

\begin{proof}
The claimed properties of $\ell^\varrho$ are trivial. 
From Proposition~\ref{prop:ellvector}  with the setting 
\begin{equation}\label{SettingRelax}
  u = \varrho \quad\mbox{with}\quad n=1 \,,\quad  u_0 = \varrho\,,\quad  A = \frac{1}{\t} 
  \quad \mbox{and}\quad f= 0\,,
\end{equation}
we infer that $\ell^\varrho$ satisfies condition~(\ref{condellmatrix}), i.e. $\T$ is the stopping time of 
the relaxation. That $\varrho$ is $C^m$ on $(0,\infty)$ follows from the solution formula in 
Proposition~\ref{prop:ellvector}.  
\end{proof}

Frequently, $\rho$ is decreasing and the control term $\ell^\varrho$ should be such that the error 
$\|\rho - \varrho\|$ is minimal with respect to some appropriate norm $\|\cdot \|$. For example, this can be 
achieved if the control term is chosen to be zero (or very small) on the interval $(0,\T_1)$ for sufficiently 
large $\T_1 < \T$. Later, in Subsection~\ref{sec-MoreRelax}, we present some examples of such relaxation 
functions $\varrho$. 

Moreover, from the physical point of view, a relaxation process should have a decreasing energy function. 
Indeed, we have 

\begin{prop}\label{prop:ellrelax2}
Let $\T$, $\varrho_0$, $\t$ and $\rho$ be as in Corollary~\ref{coro:relaxnr1}. Moreover, let $\t$ be a decreasing 
continuous function and $\ell$ be a continuous function satisfying $\supp(\ell)=[0,\T]$ and 
$\int_0^\T \frac{\ell(s)}{\rho(s)} \,\d s = 1$. Then the solution of~(\ref{FullRelaxEq0}) is $C^{1}$ on $(0,\infty)$ 
and reads as follows
\begin{equation}\label{formulavarrho}
   \varrho(t) 
       = \varrho_0\,\rho(t)\,\left[ 1 
            - \int_0^t \frac{\ell(s)}{\rho(s)} \,\d s \right]
        \qquad\mbox{for}\qquad
      t\in [0,\T]\,.
\end{equation}
If in addition, $\ell$ is monotonic decreasing on $(0,\T)$ and $\t$ is such that $\t^{-1}\,\rho$ is 
monotonic decreasing on $(0,\T)$, then the kinetic energy of the process is decreasing on $(0,\T)$ 
and vanishes at $t=\T$.  
\end{prop}

\begin{proof} 
The claims about the representation of $\varrho$ follow from Proposition~\ref{prop:ellvector} 
with the setting~(\ref{SettingRelax}). 
Now for the last claim. Let $\ell$ and $\t^{-1}\,\rho$ be monotonic decreasing on $(0,\T)$. Because  
$\int_0^t \frac{\ell(s)}{\rho(s)}\,\d s$ is positive and monotone increasing to $1$ for $t\to\T_-$, we infer from 
\begin{equation}\label{formulavarrho2}
  \varrho'(t)
      = -\varrho_0\,\left\{ \frac{\rho(t)}{\t(t)}\,\left[ 1-\int_0^t \frac{\ell(s)}{\rho(s)}\,\d s \right] 
        + \ell(t) \right\}
   \quad\mbox{for}\quad t\in (0,\T)\,,
\end{equation}
that  the kinetic energy $E_{kin} \propto (\varrho')^2$ decreases in time and satisfies 
$\lim_{t\to\T-} E_{kin}(t) = 0$. As was to be shown. 
\end{proof}

We now show that a relaxation with finite stopping time satisfies an uncontrolled relaxation 
equation with a time dependent coefficient. Indeed, this fact is very useful to model relaxation 
processes with a finite stopping time (cf. Subsection~\ref{sec-MoreRelax}).

\begin{coro}\label{coro:ellrelax2}
Let $\T$, $\t$, $\varphi$, $\ell$ and $\rho$ be as in Proposition~\ref{prop:ellrelax2} and $\t_0 := \t(0)$. 
If there is an $n\in\N$ such that
$$
  \ell^{(j)}(\T) = 0  \quad \mbox{for}\quad j\in\{0,\,1,\,\ldots,\,n-1\} 
          \qquad\mbox{and}\qquad 
  \ell^{(n)}(\T) \not= 0 \,, 
$$
then there exists a continuous function $\t_\T:\R\to [0,\infty)$ satisfying 
$$
     \t_\T|_{t=0} = \t_0\,(1+\t_0\,\ell(0+))^{-1}
        \qquad\mbox{and}\qquad 
     \t_\T|_{t\geq \T} = 0
$$ 
and such that~(\ref{formulavarrho}) can be written as 
$$
 \varrho(t) 
     = \varrho_0\,\exp\left( -\int_0^t \frac{1}{\t_\T(s)}\,\d s \right)  \qquad\mbox{for}\qquad  t\in [0,\T]\,.
$$
In particular, $\varrho$ solves the problem 
\begin{equation}\label{FullRelaxEq}
     \t_\T(t)\,\varrho'(t) + \varrho(t) = 0 \quad \mbox{for}\quad t\in (0,\T)  
           \qquad\mbox{with}\qquad 
     \varrho(0) = \varrho_0\,
\end{equation}
and
$$
  \t_\T(t) 
      = \t\left[ 1 - \frac{\ell(t)}{\rho(t)\,\left[ 1- \int_0^t \frac{\ell(s)}{\rho(s)}\,\d s\right] + \ell(t)}\right] 
        \quad \mbox{for}\quad t\in (0,\T) \,.
$$ 
\end{coro}

\begin{proof}
Because $\varrho_0>0$ and $\rho$ is positive on $(0,\T)$, it follows that $\varrho$ in~(\ref{formulavarrho}) 
is positive on $(0,\T)$ and thus  $\t_\T(t) := - \frac{\varrho(t)}{\varrho'(t)}$ is well-defined, positive 
on $(0,\T)$ and satisfies the claimed formula for $t\in(0,\T)$. (If $\varrho_0<0$, then the prove is very 
similar; but we omit this case.) Without loss of generality let $\varrho_0 = 1$. 
From~(\ref{formulavarrho}) and~(\ref{formulavarrho2}), we infer $\varrho(0+) = 1$ and 
$\varrho'(0+) = -\frac{1}{\t(0)} - \ell(0)$ and thus $\t_\T(0+) = \frac{\t_0}{1+\t_0\,\ell(0)}$. 
We have 
\begin{equation*}
\begin{aligned}
   \t_\T(\T-) 
       &= - \frac{\varrho}{\varrho'}(T-) 
        = \t(\T-)\,\lim_{t\to\T-} \frac{\rho(t)\,h(t)}
                             {\rho(t)\,h(t) + \t_1(t)\,\ell(t)} \,,
\end{aligned}
\end{equation*}
where $h(t) := 1-\int_0^t \frac{\ell(s)}{\rho(s)}\,\d s$ for $t\in (0,\T)$.  
We note that $\t(\T-) \not = 0$ (cf. assumptions in Corollary~\ref{coro:relaxnr1}), $\rho^{(m)}(\T-) \not= 0$ 
for $m\in \N\cup \{0\}$, $h(\T-)=0$ and $\ell(\T-)=0$. Hence, if $\ell'(\T-) \not= 0$, then the rule of 
L'Hospital together with $h'\,\rho = -\ell$ implies 
\begin{equation*}
\begin{aligned}
   \frac{\t_\T(\T-)}{\t(\T-)}
      = \lim_{t\to\T-} \frac{\rho'(t)\,h(t) - \ell(t) }
                             {\rho'(t)\,h(t) - \ell(t) + \t'(t)\,\ell(t) + \t(t)\,\ell'(t)} 
      = 0\,.
\end{aligned}
\end{equation*}
If $\ell'(\T-) = 0$ but $\ell''(\T-) \not= 0$, then 
\begin{equation*}
\begin{aligned}
   \frac{\t_\T(\T-)}{\t(\T-)}
          = \lim_{t\to\T-} \frac{\rho''(t)\,h(t) - \rho'(t)\,\frac{\ell(t)}{\rho(t)} - \ell'(t)}
                             {\rho''(t)\,h(t) - \rho'(t)\,\frac{\ell(t)}{\rho(t)} - \ell'(t) 
                               + \t''(t)\,\ell(t) + 2\,\t'(t)\,\ell'(t) + \t(t)\,\ell''(t)} \,.
\end{aligned}
\end{equation*}
and thus $\t_\T(\T-) = 0$. By induction, it can be shown that $\t_\T(\T-) = 0$ holds for all cases 
stated in the corollary. In summary, we see that $\t_\T$ is continuous and non-negative if we set 
$\t_\T(t) := 0$ for $t \geq \T$. Finally, the representation formula for $\t_\T$ follows from 
calculus and $\t_\T(t) = - \frac{\varrho(t)}{\varrho'(t)}$.
\end{proof}

\subsection{A control problem for dissipative oscillation}

Now we shortly apply our general results to the case of dissipative oscillations.  

In the following, we say that two function $f,\,g:(0,\T)\to\R$ are \emph{orthogonal} if 
$$
     \langle f,g \rangle := \int_0^\T f(s)\,g(s)\d s = 0   \,.
$$

\begin{coro}\label{exam:oscillnr1}
Let $\T\in (0,\infty]$, $\varphi,\,\psi\in\R$, $\omega_0>0$ and $\t_0>0$ be such that 
$\omega := \sqrt{\omega_0^2 - \frac{1}{4\,\t_0^2}} \geq 0$, 
$\rho := e^{-\frac{\cdot}{2\,\t_0}}$
$\xi := \frac{\sin(\omega\,\cdot)}{\omega}$ and $\eta := \mathcal{D}\, \xi$, where 
$\mathcal{D} := \frac{\d }{\d t} + \frac{1}{2\,\t_0}\,\Id$. 
Moreover, let $h_1$ and $h_2$ be orthogonal to $\eta$ and $\xi$, respectively. 
Then the oscillation satisfying
\begin{equation}\label{exam_oscillnr1_eqforv}
\begin{aligned}
    &v'' + \frac{v'}{\t_0} + \omega_0^2\,v = -\varphi\,\ell^\varphi - \psi\,\ell^\psi 
            \quad\mbox{on}\quad (0,\infty) \quad\mbox{with}   \\ 
    &v(0)=\varphi,\,\quad v'(0)=\psi \,,
\end{aligned}
\end{equation}
stops at the time instant $\T$ if 
$$
  \ell^\varphi := -\frac{\rho\,h_1}{\langle \xi, h_1 \rangle}  
               \qquad\mbox{and}\qquad 
  \ell^\psi := \frac{\rho\,h_2}{\langle \eta, h_2 \rangle} \,.
$$  
\end{coro}

\begin{proof} 
The oscillation problem is equivalent to 
$$
     \left( \begin{array}{c} v' \\ v'' \end{array}  \right) 
       + \left( \begin{array}{c} -v' \\ b\,v + a\,v' \end{array}  \right)
            = -\ell\,u_0 \,
                    \qquad\mbox{with}\qquad
     \ell 
               = \left( \begin{array}{cc} 0 & 0 \\ \ell^\varphi & \ell^\psi \end{array}  \right) \,,
$$ 
where $a := \frac{1}{\t_0}$ and $b := \omega_0^2$. That is to say, it is equivlent to the problem in Proposition~\ref{prop:ellvector} for the setting
$$
    u = (v,v')^T  \,,\quad  
    u_0 = (\varphi,\psi)^T \,,\quad 
    A = \left( \begin{array}{cc} 0 & -1 \\ b & a \end{array}  \right)  \quad \mbox{and}\quad 
    f = (0,0)^T \,.
$$
Hence if $\ell^\varphi=\ell^\psi=0$ then 
\begin{equation*} 
\begin{aligned}
 v(t) 
    &= \left[ \varphi\,\cos(\omega\,t) 
             + \left(\psi+\frac{\varphi}{2\,\t_0}\right)\,\frac{\sin(\omega\,t)}{\omega}\right]\,e^{-\frac{t}{2\,\t_0}} 
     = \left[ \varphi\,\eta(t) + \psi\,\xi(t) \right]\,\rho(t)
\end{aligned}
\end{equation*}
from which we infer  
$$
    G = \rho\,\left( \begin{array}{cc} \eta & \xi \\  \eta'  & \xi'  \end{array} \right) 
                  \qquad\mbox{and}\qquad
    G^{-1} = \rho^{-1}\,\left( \begin{array}{cc} \xi' & -\xi \\  -\eta'  & \eta  \end{array} \right) \,, 
$$ 
where we have used that $\eta\,\xi'-\eta'\,\xi = 1$. As a consequence, condition~(\ref{condellmatrix}) is 
equivalent to 
$$
    \langle \rho^{-1}\,\eta, \ell^\varphi \rangle = \langle \rho^{-1}\,\xi, \ell^\psi \rangle = 0
                      \qquad\mbox{and}\qquad
   - \langle \rho^{-1}\,\xi, \ell^\varphi \rangle = \langle \rho^{-1}\,\eta, \ell^\psi \rangle =  1  \,,
$$
But this shows that the matrix function $\ell$ defined as in the Corollary satisfies 
condition~(\ref{condellmatrix}), which concludes the proof. 
\end{proof}

\section{More about Relaxations}
\label{sec-MoreRelax}

In this section, we discuss further properties of relaxations (with finite stopping times) and present 
several concrete examples. 
Of course, if the stopping time $\T$ is sufficiently large, then some of these examples are very good 
approximations of the classical relaxation function $\rho: t \mapsto \exp\left( -\frac{t}{\t_0} \right)$, 
where $\t_0>0$ is a constant.  
We start with a simpe example.

\begin{exam}\label{exam:relaxneq1}
Let $\T>0$, $\mu>0$, $a(t):= \frac{1}{\t_\T} := \frac{\mu\,\T}{\T-t}$ for $t\in[0,\T)$, then 
$$
         \varrho(t) := \left( 1 -\frac{t}{\T}\right)^{\mu\,\T}  \qquad\mbox{for}\qquad t\in [0,\T)
$$
solves the relaxation problem~(\ref{FullRelaxEq}) with $\varrho(0)=1$. Moreover, we see that 
\begin{itemize}
\item if $\mu\,\T > 1$ then $\varrho'(\T-) = 0$,  

\item if $\mu\,\T = 1$ then $\varrho'(\T-) = -\frac{1}{\T} < 0$ and 

\item if $\mu\,\T \in (0,1)$ then $\varrho'(\T-) = -\infty$.
\end{itemize}
The energy function of the relaxation process is given by 
$E \propto  \frac{1}{2}\,(\varrho')^2(t) = \mu^2\,\left( 1 -\frac{t}{\T}\right)^{2\,\mu\,\T-2}$ and therefore 
the energy is decreasing if and only if $\mu\,\T\geq 1$. 
\end{exam}

That the energy function of $\varrho$ from the previous example is decreasing if and only if 
$\mu\,\T\geq 1$ can also be concluded from the following lemma. 

\begin{lemm}\label{lemm:eneryrelax}
Let $\T\in (0,\infty]$, $\t_\T$ be as in Corollary~\ref{coro:ellrelax2} and $a:=\t_\T^{-1}$. Moreover, let $a$ 
be continuously differentiable on $(0,\T)$. Then the energy of a process modeled by~(\ref{FullRelaxEq}) 
has a positive and decreasing energy if and only if  
$$
    a^2(t) \geq a'(t) \qquad\mbox{for}\qquad  t\in(0,\T)\,.
$$
\end{lemm}

\begin{proof}
We note that the solution $\varrho$ of~(\ref{FullRelaxEq}) satisfies $\varrho' =- a\,\varrho$ and 
$\varrho''=(a^2-a')\,\varrho$ and that the energy function of the process $\varrho$ is given by 
$E(t) \propto \frac{1}{2}\,(\varrho')^2$. Hence it follows that $E' = \varrho'\,\varrho'' = -a\,(a^2-a')\,\varrho^2$, 
i.e. $E' \leq 0$ if and only if $a^2(t) \geq a'(t)$ for $t\in (0,\T)$. 
\end{proof}

We now give some useful examples of relaxation processes with finite stopping time and decreasing 
energy function. We remind the reader that usually $\T$ is large and thus $1/\T$ is small.

\begin{exam}
Let $\T>0$ and $a(t) := a_0 + a_1\,\tan(\frac{\pi\,t}{2\,\T})$ for $t\in(0,\T)$ with positive constants $a_0$ 
and $a_1$. If 
$$
   a_1 \geq \frac{\pi}{2\,\T}
          \qquad\mbox{and}\qquad   
   a_0^2 \geq \frac{\pi\,a_1}{2\,\T}  \,, 
$$
then $\varrho(t) := \varphi\,\exp\left( -\int_0^t a(s)\,\d s \right)$ with some $\varphi \not= 0$ describes a 
relaxation process with relaxation time function $\t_\T = a^{-1}$ that stops at time $\T$ and has decreasing 
energy function. Indeed, condition $a^2\geq a'$ from Lemma~\ref{lemm:eneryrelax} is equivalent to
\begin{equation*}
\begin{aligned}
   (a_0^2 - a_1\,b_0) 
         + 2\,a_0\,a_1\,\tan(b_0\,t) 
         + a_1\,(a_1 - b_0)\tan^2(b_0\,t) \geq 0  \quad\mbox{for}\quad t\in (0,\T)\,
\end{aligned}
\end{equation*}
with $b_0 := \frac{\pi}{2\,\T}$ 
which is satisfied if and only if $a_1 \geq \frac{\pi}{2\,\T}$ and $a_0^2 \geq \frac{\pi\,a_1}{2\,\T}$.

Is it possible to choose the constant $a_j$ such that $u'(0)=0$ and the energy is decreasing? No. From 
the conditions on $a_j$, we see that $a_1>0$ and $a_0>0$ and thus $\varrho'(0) = -a_0\,\varphi \not= 0$.
\end{exam}

Further reasonable examples for relaxations can be inferred from Corollary~\ref{coro:relaxnr1} together with 
the following control terms  
\begin{equation}\label{ellrelaxmodsn}
  \ell_n(t) := \frac{(\T-t)^n}{\int_0^\T  e^{s}\,(\T-s)^n \,\d s} \,\chi_{[0,\T]}(t) 
   \qquad\mbox{for}\qquad n\in\N 
\end{equation}
which is $C^n$ on $(0,\infty)$ and thus $\varrho_n\equiv \varrho$ is also $C^n$ on $(0,\infty)$. For the case 
$n=\infty$, we consider 
\begin{equation}\label{ellrelaxmodsinfty}
  \ell_\infty(t) 
            := \frac{\exp\left( -\frac{t^2}{\T^2-t^2} \right)}
                    {\int_0^\T  e^{s}\,\exp\left( -\frac{s^2}{\T^2-s^2} \right) \,\d s} \,\chi_{[0,\T]}(t)  
\end{equation}
which is $C^\infty$ on $(0,\infty)$ and therefore $\varrho_\infty\equiv \varrho$ is $C^\infty$ on $(0,\infty)$. 
This leads to the following definition.

\begin{defi}\label{defi:relaxn}
Let $\T>>1$,
$$
    p_n(t) := (\T-t)^n \quad (n\in\N) 
      \quad\mbox{and}\quad  
    p_\infty(t) := \exp\left( -\frac{t^2}{\T^2-t^2} \right) \quad\mbox{for}\quad t\in\R 
$$ 
and $I_n(t) := \int_0^t  e^{s}\,p_n(s) \,\d s$ for $t\in [0,\T]$. Then we define the function $\varrho_n:\R\to\R$ by
$$
   \varrho_n(t) := \left\{ \begin{array}{ll} 
                       e^{-t}\,\left[ 1- \frac{ I_n(t) }{ I_n(\T) } \right]\quad 
                                             & \mbox{for \quad $t\in [0,\T]$} \\
                       0  \quad &  \mbox{for \quad $t\in \R\backslash [0,\T]$}
                \end{array} \right.  \,.
$$
\end{defi}

\begin{coro}\label{coro:rhon}
Let $n\in\N \cup \{\infty\}$ and $\varrho_n$ be defined as in Definition~\ref{defi:relaxn}. Then $\varrho_n$ 
is decreasing, convex and $C^n$ on $(0,\T)$. Moreover, it satisfies $\varrho_n^{(k)}(\T-) = 0$ for 
$k\in \{0,\,1,\,\ldots,n-1\}$ and its energy function $E$ is decreasing. 
\end{coro}

\begin{proof}
For $n\in\N$ and $t\in(0,\T)$, we have
\begin{equation}\label{derivrhon}
\begin{aligned}
    \varrho_n^{(k)}(t) 
            &= (-1)^k\,\varrho_n(t) 
              + \sum_{j=0}^{k-1} (-1)^{k-j}\,\frac{p_n^{(j)}(t)}{I_n(\T)}  \qquad (k\in\N)\,
\end{aligned}
\end{equation} 
and hence $\varrho_n$ is $C^\infty$ on $(0,\T)$, decreasing and $\lim_{t\to\T-} \varrho_n^{(k)}(t) = 0$ for 
$k\in \{0,\,1,\,\ldots,n-1\}$, due to $p_n^{(j)}(t) = \frac{(-1)^j\,n!}{(n-j)!}\,p_{n-j}(t)$. Moreover, it follows 
that $\varrho_n''(t) >0$ for $t\in (0,\T)$, i.e. $\varrho_n$ is convex on $(0,\T)$. Because  
\begin{equation}\label{helpa}
    a(t) = -\frac{\varrho_n'(t)}{\varrho_n(t)} = 1 + \frac{p_n(t)\,e^{t}}{I_n(\T) - I_n(t)}\,,
\end{equation} 
condition $a^2\geq a'$ from Lemma~\ref{lemm:eneryrelax} is equivalent to
$$
    I(\T) - I(t) + [p_n(t)-p_n'(t)]\,e^t \geq 0  \quad\mbox{for}\quad t\in (0,\T) \,,
$$
and thus the energy is decreasing on $(0,\T)$. Here we have used that $-p_n'$ is positive on $(0,\T)$. \\
Now let $n=\infty$. For this case, formulae~(\ref{derivrhon}) and~(\ref{helpa}) also hold. From this and
\begin{equation}\label{helpexp}
   \lim_{t\to\T-} \frac{1}{(\T^2-t^2)^l}\,\exp\left( -\frac{t^2}{\T^2-t^2} \right) = 0 
        \quad\mbox{for}\quad l\in\N\,,
\end{equation}
it follows that $\varrho_\infty$ is $C^\infty$ on $(0,\T)$, decreasing and $\lim_{t\to\T-} \varrho_\infty^{(k)}(t) = 0$ 
for $k\in \N$. In particular, $\varrho_\infty$ is $C^\infty$ on $(0,\infty)$. Moreover, $\varrho_\infty''>0$ 
and $a^2\geq a'$ hold on $(0,\T)$, due to~(\ref{derivrhon}),~(\ref{helpa}) and $p_\infty' < 0$ on $(0,\T)$. 
That is to say, $\varrho_\infty$ is convex and $E$ is decreasing on $(0,\T)$. 
As was to be shown.  
\end{proof}

We conclude this section with explicite formulae for the relaxations $\varrho_1$ and $\varrho_2$ from 
Definition~\ref{defi:relaxn} as well as the coefficients $a_1$ and $a_2$ (cf.~(\ref{helpa})). 

\begin{exam}\label{exam:relaxneq2}
We now consider the relaxation functions defined in Definition~\ref{defi:relaxn} for $n=1$ and $n=2$.  
From $I_1(t) = (\T+1-t)\,e^t - (\T+1)$ and 
$I_2(t) = (\T^2+2\,\T+2+t^2-(2\,\T+2)\,t)\,e^t - (\T^2+2\,\T+2)$, we infer
$$
   \varrho_1(t) 
         = \frac{ e^{\T-t} + t - (\T+1) }{e^\T - (\T+1) } \,,
$$
$$
   \varrho_2(t) 
         = \frac{ 2\,e^{\T-t} - t^2 + (2\,\T+2)\,t - (\T^2+2\,\T+2) }
                                      { 2\,e^\T - (\T^2+2\,\T+2) }   \,,
$$
$$
 a_1(t) 
     = \frac{e^{\T-t} - 1}{e^{\T-t} - (\T+1-t)}  
$$
and
$$
 a_2(t) 
     = \frac{ 2\,(e^{\T-t} + t - (\T+1) ) }{ 2\,e^{\T-t} - (\T^2+2\,\T+2 +t^2-(2\,\T+2)\,t) } 
$$
\end{exam}

\section{Diffusion with finite front speed}
\label{sub-Diff}

First, we generalize the notation of Gaussian functions, which we denote by $\G$, and then, we show that 
they satisfy the classical diffusion equation with a control term that guarantees a diffusion front 
with finite speed. Moreover, we focus on one space dimension, but the generalization to higher dimensions 
is straight forward. Some basic facts and applications of diffusion can be found  in~\cite{Harris79,FeWa80,KitKro93,Wei98a,Soi99,GuiMorRya04,KiSrTr06,ScGrGrHaLe09,Ko13}.

\begin{defi}\label{defi:modelGcure}
Let $\T\in(0,\infty)$, $\varrho:\R\to\R$ be a relaxation function with $\supp(\varrho)=[0,\T]$, i.e. 
$\varrho$ solves problem~(\ref{FullRelaxEq0}) for given relaxation time $\t_0$ and control term $\ell$ 
satisfying $\int_0^\T e^s\,\ell(s)\,d s = 1$. 
Moreover, let $N := \frac{ \sqrt{\pi} }{ \int_{-\T}^\T \varrho_n(s^2)\,\d s }$ and $R:[0,\infty)\to\R$ be a 
positive increasing $C^1-$function satisfying 
$$
  \lim_{t\to\infty} R(t) = \infty 
     \qquad\mbox{and}\qquad  
  R'(t) \in (0,c_0] \quad\mbox{for}\quad  t\in [0,\infty)\,, 
$$
where $c_0$ is a positive constant. Then we call the function 
\begin{equation*}
\begin{aligned}
   \G(\cdot,t) := \left\{ \begin{array}{ll} 
                  \frac{N}{\sqrt{\pi}}\,\frac{\partial s}{\partial x}(\cdot,t)\,\varrho\left( s^2(\cdot,t) \right)  \quad 
                                             & \mbox{for \quad $t>0$} \\
                   0  \quad &  \mbox{for \quad $t \leq 0$}
                \end{array} \right.  
   \quad \mbox{with}\quad   s :=  \frac{\sqrt{\T}\,x}{R}
\end{aligned}
\end{equation*} 
a \emph{generalized Gaussian function} induced by $\varrho$ and $s$. We call $R(t)$ and $R'(t)$ the radius and the speed 
of the diffusion front at time $t$, respectively. 
\end{defi}

It is easy to see that $G(\cdot,t)$ lies in $C^\infty(\R) \cap L^1(\R)$ for $t>0$ and 
$\supp(G(\cdot,t)) = [-R(t),R(t)]$. The latter claim follows at once from $G(x,t) =0$ if and only if 
$0\leq \frac{\T\,x^2}{R^2} \leq \T$. Because
$$
   \int_\R G(x,t)\,\d x = \frac{N}{\sqrt{\pi}}\,\int_\R \varrho(s^2)\,\frac{\partial s}{\partial x}\,\d x = 1\,,
$$ 
the function $x\mapsto G(x,t)$ (for fixed $t>0$) is a probability density.

\begin{rema}
What is the front speed $c_F$ of the process $\G$ defined as in Definition~\ref{defi:modelGcure}? Let us consider 
a particle propagating at the front of the diffusing matter, then the position $R$ of this particle at time $t$ is 
given by $R(t)$ or $-R(t)$, due to $\supp(G(\cdot,t)) = [-R(t),R(t)]$, and thus $c_F(t) = R'(t)$.
\end{rema}

\begin{exam}\label{exam:Gcure0}
If $\t_0 = 1$, $\ell = 0$ and $R(t) = \sqrt{ 4\,T\,D_0\,t }$, then $\G$ reduces to the fundamental solution of 
the standard diffusion equation with diffusivity $D_0>0$. For this case, we have 
$$
    \supp(G(\cdot,t)) = \R  \qquad\mbox{and}\qquad R'(t) = \sqrt{\frac{4\,D_0}{t}} 
         \qquad\mbox{for}\qquad t>0\,.
$$
Because $\ell=0$, we have $\supp(G(\cdot,t)) = \R$ and thus the function $R(t)$ does not describe the front of 
the diffusion process; there is no front for this model process. 
\end{exam}

\begin{exam}\label{exam:Gcure1}
If $R(t) = \sqrt{ 4\,T\,\frac{D_0}{\t_0}\,t }$ with positive constants $D_0$ and $\t_0$, then the process $\G$ 
in Definition~\ref{defi:modelGcure} has a front that propagates with speed $R'(t) = \sqrt{\frac{4\,D_0}{\t_0\,t}}$. 
This front speed is bounded for \emph{sufficiently large times} and satisfies $\lim_{t\to 0+} R'(t) = \infty$. 
\end{exam}

\begin{exam}\label{exam:Gcure2}
An example for $R$ satisfying the assumptions in Definition~\ref{defi:modelGcure} is given by 
\begin{equation}\label{examRGcure2}
\begin{aligned}
   R(t) := \left\{ \begin{array}{ll} 
                   \sqrt{ \frac{D_0\,\T}{t_0^3} }\,\left( -t^2 + 3\,t_0\,t\right) \quad 
                                             & \mbox{for \quad $t\in[0,t_0]$} \\
                   \sqrt{4\,\T\,D_0\,t}  \quad &  \mbox{for \quad $t > t_0$}
                \end{array} \right.  \,,
\end{aligned}
\end{equation} 
where $\t_0=1$, $t_0$ and $c_0$ are related by $c_0 = 3\,\sqrt{D_0\,\frac{\T}{t_0}} < \infty$. For this case 
$\G$ has a finite front speed $\leq c_0$. In contrast to the previous example, we have modified the 
function $R$ for small time values such that the front speed is always bounded by $c_0$.  
\end{exam}

Let $\G$, $R$, $\ell$ and $s$ be as in Definition~\ref{defi:modelGcure} and $D_0>0$ be a constant. For example, 
$D_0$ and $R$ may be related by $R(t_0)=\sqrt{4\,\T\,D_0\,t_0}$ for some $t_0>0$ (cf. Example~\ref{exam:Gcure2}). 
Note according to our assumptions $|R(t)| \leq c_0\,t$ is true.  
We now consider the problem of modeling a control function $\L:\R^2\to\R$ (via the help of $\ell$) such that 
the solution of
\begin{equation}\label{DiffProblem}
\begin{aligned}
      &\frac{\partial \G}{\partial t}(x,t) - D_0\,\frac{\partial^2 \G}{\partial x^2}(x,t)  
            = \delta(t)\,\delta(x) - \L(x,t)  
         \quad\mbox{on}\quad  \R^2\, \qquad\mbox{with} \\ 
     &\G|_{t<0} = 0
\end{aligned}
\end{equation}
has at each instant of time compact support in $[-R(t),R(t)]$ and a finite front speed. We see that $\G$ is \emph{causal} 
in the sense that $G(x,\cdot)$ vanishes on $(-\infty,0)$ for each $x\in\R$. Hence, if $f$ is a causal distribution,  
then $u := G *_{\x,t} f$ exists and is a causal distribution\footnote{Note that $\G(\cdot,t)$ has compact support in 
$\R$ for $t\in\R$.} (cf. ~\cite{GaWi99,Ho03}). It is easy to see that 
$u$ satisfies 
\begin{equation}\label{DiffContProblem}
\begin{aligned}
      \frac{\partial u}{\partial t} - D_0\,\frac{\partial^2 u}{\partial x^2}  
            = f - \L *_{x,t} f  
         \quad\mbox{on}\quad  \R^2\, \quad\mbox{with}\quad  u|_{t<0} = 0
\end{aligned}
\end{equation}
and $u(\cdot,t)$ has compact support if $f(\cdot,t)$ has compact support for $t>0$. In othere words, the control 
term $\L *_{x,t} f$ guarantees the causality condition 
$$
          u(x,t) = 0 \qquad\mbox{for}\qquad  \frac{|x|}{c_1} > t \qquad \mbox{for some $c_1\in (0,\infty)$}
$$
for the diffusion process $u$. We now determine $\L$ as a function of $\varrho$, $\ell$ and $\ell'$, where 
$\ell$ is actually considered as the extension $\tilde\ell$ of $\ell$ satisfying 
$\tilde \ell(t) = 0$ for $t<0$ and $\tilde \ell(0) := \tilde\ell(0+)$.

\begin{prop}\label{prop:diffGcure}
Let $\T$, $\ell$, $R$, $s$ and $\G$ be defined as in Definition~\ref{defi:modelGcure} and $D_0>0$ be a constant. 
Then $\G$ solves problem~(\ref{DiffProblem}) with
\begin{equation*}
\begin{aligned}
       \L := - a_0\,\varrho(s^2) - a_1\,\ell(s^2) - a_2\,\ell'(s^2) 
\end{aligned}
\end{equation*}
where
\begin{equation*}
\begin{aligned}
   &\qquad\qquad\qquad
     a_0 := \frac{N}{\sqrt{\pi}}\,\frac{\sqrt{\T}}{R^2}\,   
                                          \left( \frac{2}{\t_0}\,s^2 - 1 \right)
                                          \left( R' - \frac{2\,\T\,D_0}{\t_0\,R} \right)\,, \\
   & a_1 := \frac{N}{\sqrt{\pi}}\,\frac{\sqrt{\T}}{R^2}\,\left[   
                                           2\left( R' - \frac{2\,\T\,D_0}{\t_0\,R} \right)\,s^2  
                                         + \frac{2\,\T\,D_0}{R} \right]
    \quad\mbox{and}\quad
    a_2 := \frac{N}{\sqrt{\pi}}\,\frac{4\,\T^{3/2}\,D_0}{R^3}\,s^2\,.
\end{aligned}
\end{equation*} 
If $R$ is as in Example~\ref{exam:Gcure2}, then $a_0(\cdot,t)=0$ for $t\geq t_0$. 
\end{prop}

\begin{proof}
Let $x\in\R$ and $t>0$. Differentiation of $\G$ with respect to $t$ and $x$ yields
\begin{equation}\label{helprelG}
\begin{aligned}
    &\frac{\partial \G}{\partial t}(x,t) 
             = - \frac{N}{\sqrt{\pi}}\,\frac{\sqrt{\T}\,R'(t)}{R^2(t)}\,\left[ 
                                        \varrho\left( s^2(x,t) \right) 
                                         + 2\,s^2(x,t)\,\varrho'\left( s^2(x,t) \right)  \right]  \\
    &\frac{\partial^2 \G}{\partial x^2}(x,t) 
             = \frac{N}{\sqrt{\pi}}\,\frac{2\,\T^{3/2}}{R^3(t)}\, \left[ 
                                     2\,s^2(x,t)\,\varrho''\left( s^2(x,t) \right) 
                                       + \varrho'\left( s^2(x,t) \right) \right]
\end{aligned}
\end{equation}
for $x\in\R$ and $t>0$. From  this together with
\begin{equation*}
      \varrho'(s^2) = - \frac{\varrho(s^2)}{\t_0} - \ell(s^2) 
            \quad\mbox{and}\quad
      \varrho''(s^2) = \frac{\varrho(s^2)}{\t_0^2} + \frac{\ell(s^2)}{\t_0} - \ell'(s^2)
             \quad\mbox{for}\quad s>0\,,
\end{equation*}
it follows that equation~(\ref{DiffProblem}) with the claimed control function $\L$ is true for $t>0$. 
Because $\G$ and $\ell$ vanishes for negative time,~(\ref{DiffProblem}) also holds for $x\in\R$ and $t<0$. 
From $\int_\R \G(x,t)\,\d x = 1$ for $t>0$, it follows that $\lim_{t\to 0+} \G(\x,t) = \delta(x)$, 
which justifies the right hand side term $\delta(t)\,\delta(x)$ in~(\ref{DiffProblem}). 
This concludes the proof. 
\end{proof}

\begin{rema}
a) The above proposition also holds if $D_0$ is "replaced" by a positive time dependent diffusivity $D$. More precisely, 
\begin{itemize}
\item if $D_0$ in~(\ref{DiffProblem}) is replaced by $D$, then the formulae for $a_0$, $a_1$ and $a_3$ 
      (as well as $\L$) remain true if $D_0$ is also replaced by $D$, 

\item Example~\ref{exam:Gcure2} remains true if $\sqrt{4\,\T\,D_0\,t}$ in (\ref{examRGcure2}) is replaced by 
      $\sqrt{4\,\T\,\tilde D(t)}$, where $\tilde D(t) := \int_0^t D(z)\,\d z$ with $\tilde D'(t_0) = D_0$, and

\item the control problem~(\ref{DiffContProblem}) with the respective control function $\L$ remains true for 
      forcing terms $f$ of the form $\delta(t)\,F(x)$ (initial value problems), where $F$ is a distribution on $\R$.  
\end{itemize}

\end{rema}

For completeness, we include the following

\begin{coro}
Let $\varrho$, $R$, $s$ and $\G$ be defined as in Definition~\ref{defi:modelGcure}. There exists a nonnegative 
function $D:\R^2\to\R$ such that $\G$ satisfies
\begin{equation}\label{DiffProblem2}
\begin{aligned}
      \frac{\partial \G}{\partial t}(x,t) - D(x,t)\,\frac{\partial^2 \G}{\partial x^2}(x,t)  
            = \delta(t)\,\delta(x)   
         \quad\mbox{on}\quad  \R^2\, \quad\mbox{with} \quad\G|_{t<0} = 0\,.
\end{aligned}
\end{equation}
\end{coro}

\begin{proof}
According to Corollary~(\ref{coro:ellrelax2}) there exists a 
nonnegative $C^1$-function $\t:(0,\infty)\to\R$ such that $\varrho'(t) = - \frac{\varrho(t)}{\t(t)}$ holds 
for $t > 0$. From~(\ref{helprelG}) and 
$$
      \varrho'(s^2) = - \frac{\varrho(s^2)}{\t(s^2)} 
            \quad\mbox{and}\quad
      \varrho''(s^2) = \frac{\varrho(s^2)}{\t^2(s^2)}\,(1 + \t'(s^2))
             \quad\mbox{for}\quad s>0\,,
$$
we infer 
\begin{equation*}
\begin{aligned}
   D := \left\{ \begin{array}{ll} 
                   \frac{ R\,R'\,\t(s^2)\,[\t^2(s^2) - 2\,s^2] }
                        { 2\,\T\,[2\,s^2\,[1+\t'(s^2)] - \t(s^2))] } \quad 
                                             & \mbox{on \quad $K_R$} \\
                   0  \quad &  \mbox{on \quad $\R^2 \backslash K_R$}
                \end{array} \right.  \,,
\end{aligned}
\end{equation*} 
where $K_R := \{(x,t) \in\R^2\,|\, |x| < R(t)\}$. This proves the claim. 
\end{proof}

\section{Control problems for dissipative waves with "local finite stopping times"}
\label{sec-DissWave}

In this section, we are interested in dissipative waves (, which have \emph{finite stopping times local in space}) 
and control problems related to these type of waves. By a wave with finite stopping times local in space we mean 
that the oscillation (caused by the wave) in an arbitrary  point in space stops after a finite time period. 
For basic facts about dissipative waves, 
we refer to~\cite{NacSmiWaa90,Sz94,Sz95,KiFrCoSa00,WaHuBrMi00,We00,HanSer03,CheHolm04,WaMoMi05,PatGre06,
KoSc12,Ko14}.)\\

The following definition is strongly motivated from our causality analysis of dissipative waves 
in~\cite{KoScBo10,KoSc12}. 
There \emph{causality} of a dissipative wave $G$ means that the front speed of $G$ is finite, which is stronger 
than $G|_{t<0} =0$. If the latter condition holds, then $G$ is called a \emph{causal} 
distribution. 
By $\F(f)$ and $\hat f$ we denote the Fourier transform of $f$ w.r.t time and by $\F^{-1}(g)$ and $\check g$ 
we denote the inverse Fourier transform of $g$ (cf. Appendix).

\begin{defi}\label{defi:modeldiiswave}
Let $\T\in (0,\infty)$, $\K$ be a real valued distribution satisfiying (i) $\K(0,t) = \delta(t)$,
(ii) $\supp(\K(1,\cdot)) = [0,\T]$ and (iii) $\hat\K(R,\omega) = \exp(-\alpha(\omega)\,R)$ ($R \geq 0$), 
where $\alpha:\R\to\C$ has nonnegative real part. 
Moreover, let $G_0$ denote the fundamental solution of the standard 
wave equation with sound speed $c_0$. Then we call 
\begin{equation*}
   \G(x,t) := (G_0 *_t \K)(x,t) = \frac{\K\left(|x|,t-\frac{|x|}{c_0}\right)}{4\,\pi\,|x|}  
    \quad\quad x\in\R^3,\,t\in\R\,,
\end{equation*}
a \emph{dissipative spherical wave} $\G$ with (frequency dependent \emph{attenuation law} $\alpha$, 
front speed $c_0$ and) local finite stopping times. \\ 
If $K$ satisfies (i), (ii) and (iii) for $\T=\infty$, then $G := (G_0 *_t K)$ is called a 
\emph{dissipative spherical wave}. In this case, the attenuation law is denoted by $\beta$ and not by $\alpha$. 
\end{defi}

According to Proposition~\ref{prop:alpha2} in the Appendix, a wave defined by 
Definition~\ref{defi:modeldiiswave} satisfies 
$$
     \supp(\K(R,\cdot)) = [0,R\,\T]  \qquad\mbox{for}\qquad    R \geq 0\,,
$$
i.e. $(\K(R,\cdot))_{R\geq 0}$ is a semigroup with $\K(0,t) = \delta(t)$ and linearly increasing 
support. 
Moreover, we note that property (ii) in the above definition implies that the oscillation $\G(x,\cdot)$ at position $x$ 
starts at time instant $\T_0(|x|) := \frac{|x|}{c_0}$ and stops at time instant $\T_0(|x|) + |x|\,\T$, i.e. 
the front speed of $\G$ is $c_0$ and the oscillation $\G(x,\cdot)$ ("local wave at $x$") takes place during the 
time period $|x|\,\T$. Hence we say that $\G$ has \emph{local finite stopping times}. 
As shown in~\cite{KoSc12}, $\G$ satisfies the initial conditions 
$$
     \G(x,0+) = \delta   \qquad\mbox{and}\qquad\frac{\partial \G}{\partial t}(x,0+)=0
$$
and is the unique solution of the integro-differential equation
\begin{equation}\label{disswaveeq}
\begin{aligned}
    &\left( D_\alpha + \frac{1}{c_0^2}\,\frac{\partial }{\partial t} \right)^2\,\G 
               - \Delta \G 
             = \delta(t)\,\delta(x)
       \quad\mbox{on}\quad \R^4   \quad\mbox{with}\quad \G|_{t<0} = 0\,,
\end{aligned}
\end{equation}
where $D_\alpha(g) := \check\alpha *_t g$. 

According to Proposition~\ref{prop:alpha1} and Proposition~\ref{prop:alpha2} in the Appendix, each relaxation 
function $(x,t)\mapsto \varrho(x,t)$ (satisfying $\|\varrho\|_{L^1(\R)} \leq \frac{1}{\sqrt{2}}$) with infinite 
or finite stopping time is of the form $(x,t)\mapsto \F^{-1}(e^{-\alpha\,|x|})(t)$, i.e. they can be used as 
models for $K$ or $\K$, respectively.  
In Example~\ref{exam:oscillrelax} in the Appendix, it is shown that the oscillation 
$K(t) := e^{-a_0\,t}\,\cos(\omega_0\,t)\,H(t)$ with $a_0\geq 1$ and $\omega_0\in (0,a_0)$ can be written 
as $\F^{-1}\left(e^{-\alpha}\right)$ with some attenuation law $\alpha$. Of course, this is not true 
for any oscillation. We now give another example of such an oscillation, but with a (large) finite stopping 
time. For a generalization of this proposition, we refer to Proposition~2 in~\cite{Ko19b}.

\begin{prop}\label{prop:oscillrelax}
Let $\T\in (0,\infty)$, $a_0 \geq 1$, $\omega_0 > 1$, $b_0 := \omega_0^2 + a_0^2$,  and  
$$
       \K(t) := \varrho(t)\,\cos(\omega_0\,t) \qquad\mbox{for}\qquad t\in\R\,, 
$$ 
where $\varrho$ is defined by\footnote{Note that $\ell_1$ is a control function for the relaxation equation 
with coefficient $a_0=1$. Here we used $\delta(a_0\,t) = |a_0|\,\delta(t)$.}
$$
     \varrho'(t) + a_0\,\varrho(t) = \delta(t) - a_0\,\ell_1(a_0\,t) \quad\mbox{for}\quad t\in\R 
     \quad\mbox{with}\quad \varrho|_{t<0}=0\,
$$ 
and $\ell_1$ defined as in~(\ref{ellrelaxmodsn}). 
Then $\K$ is the unique solution of 
$$
         \K''(t) + 2\,a_0\,\K'(t) + b_0\,\K(t) = f(t) \quad\mbox{for}\quad t\in\R 
        \quad\mbox{with}\quad \K|_{t<0}=0\, 
$$
and 
\begin{equation*}
\begin{aligned}
   f(t) := &  [\delta'(t) - a_0\,\ell_1'(a_0\,t)]\,\cos(\omega_0\,t) \\
           &+ [\delta(t) - a_0\,\ell_1(a_0\,t)]\,[a_0\,\cos(\omega_0\,t) - 2\,\omega_0\,\sin(\omega_0\,t)] \,.
\end{aligned}
\end{equation*} 
If $\T$ is sufficiently large, then there exists an attenuation law $\alpha$ such that $\hat \K = e^{-\alpha}$. 
\end{prop}

\begin{proof}
It follows straight forward that $\K$ satisfied the claimed oscillation equation with forcing term 
\begin{equation*}
\begin{aligned}
   f(t) := &  [\delta'(t) - a_0\,\ell_1'(a_0\,t)]\,\cos(\omega_0\,t) \\
           &+ [\delta(t) - a_0\,\ell_1(a_0\,t)]\,[a_0\,\cos(\omega_0\,t) - 2\,\omega_0\,\sin(\omega_0\,t)] \,.
\end{aligned}
\end{equation*} 
The details are left to the reader. The claimed form of $f$ follows from 
$$
     g(t)\,\delta(t) = g(0)\,\delta(t)  \qquad\mbox{and}\qquad
     g(t)\,\delta'(t) = g(0)\,\delta'(t) - g'(0)\,\delta(t) 
$$ 
for $g$ differentiable in $0$. \\
For the last claim, we have to show that $|\hat \K(\omega)|<1$ for all $\omega\in\R$ if $\T$ is sufficiently 
large. According to the definition of $\K$, we have $\K=\varrho_1$, where $\varrho$ is as in 
Example~\ref{exam:relaxneq2} and thus 
\begin{equation*}
\begin{aligned}
   \K(t)  
     &=  A(\T)\,e^{-a_0\,t}\,\cos(\omega_0\,t)\,H(t) 
       + B(\T)\,a_0\,t\,\cos(\omega_0\,t)\,H(t) \\
     &\quad + C(\T)\,\cos(\omega_0\,t)\,H(t) 
\end{aligned}
\end{equation*} 
with positive constants 
\begin{equation*}
\begin{aligned}
   A := \frac{e^\T}{e^\T-(\T+1)}\,,\quad 
   B := \frac{1}{e^\T-(\T+1)} \quad\mbox{and}\quad 
   C := -\frac{\T+1}{e^\T-(\T+1)} \,.
\end{aligned}
\end{equation*} 
From this representation formula and Lemma~\ref{lemm:fourier} in the Appendix, we get 
\begin{equation*}
\begin{aligned}
   \hat \K(\omega) 
        &= \frac{A(t)\,(a_0 + (-\i\,2\,\pi,\omega))}{\omega_0^2 + [a_0 + (-\i\,2\,\pi\,\omega)]^2} 
         + \frac{B(\T)\,a_0 + C(\T)\,(-\i\,2\,\pi\,\omega)}{\omega_0^2 + (-\i\,2\,\pi\,\omega)^2}  \\
        &\quad - \frac{B(\T)\,2\,a_0\,\omega_0}{[\omega_0^2 + (-\i\,2\,\pi\,\omega)^2]^2} \,.
\end{aligned}
\end{equation*} 
If $\T$ is large, then $B$ and $C$ are very close to zero and $A$ is close to one. From this, 
$a_0 \geq 1$, $\omega_0>1$ and the above representation formula of $\hat\K$, we infer 
$|\hat \K(\omega)|<1$ for all $\omega\in\R$ if $\T$ is sufficiently large. 
\end{proof}

Now we come to the first control problem. Let $\T\in (0,\infty)$ and $\beta$ be an attenuation law such that 
$K(R,t) :=\F^{-1}\left(e^{-\beta\,R}\right)$ for $R\geq 0$ satisfies 
$\supp \left( K(1,\cdot) \right) = [0,\infty)$. We are interested in control terms $\L$ 
such that the solution of 
\begin{equation}\label{contdisswaveeq}
\begin{aligned}
    \Box_\beta\,\G = \delta(t)\,\delta(x) - \L(x)
       \quad\mbox{on}\quad \R^4  \quad\mbox{with}\quad \G|_{t<0} = 0 
\end{aligned}
\end{equation}
satisfies 
\begin{equation} 
    \supp(\G(x,\cdot)) = [ T_0(|x|),T_0(|x|) + |x|\,\T )  \,,
\end{equation}
where $\Box_\beta := \left( D_\beta + \frac{1}{c_0^2}\,\frac{\partial }{\partial t} \right)^2 - \Delta$.

\begin{prop}\label{prop:contprobdisswave}
Let $\T\in (0,\infty)$, $a_0$ be a positive constant and $\varrho$ with $\supp(\varrho)=[0,\T]$ satisfy 
\begin{equation}\label{releqforwave}
      \varrho'(t) + a_0\,\varrho = \delta(t) - \ell(t) \quad\mbox{on}\quad \R \quad\mbox{with}\quad  
      \varrho|_{t<0} = 0 \,
\end{equation}
for some control function $\ell$. 
Moreover, let $\K(R,\cdot) := \F^{-1}\left( \varrho^R \right)$, 
$K(R,\cdot) := \F^{-1}\left( \rho^R \right)$, where $\rho$ solves~(\ref{releqforwave}) with $\ell=0$ and 
let $\G$, $G$ and $\beta$ be as in Definition~\ref{defi:modeldiiswave}. 
Then $\G$ is the unique solution of~(\ref{contdisswaveeq}) with control function  
\begin{equation*}
\begin{aligned}
     \L(x,\cdot) 
          := G *_t  \left[ D_\mu (2\,\D_{c_0,\beta} + D_\mu)\,\M_R  \right]\,,
\end{aligned}
\end{equation*}
where $\D_{c_0,\beta} := \frac{1}{c_0}\frac{\partial}{\partial t} + D_\beta$, 
$\mu := \F^{-1} \left( \log(1-\hat\ell) \right)$ and 
$\M_R := \F^{-1} \left( (1 - \hat\ell)^R \right)$ for $R \geq 0$.
\end{prop}

\begin{proof}
Because $\varrho = \rho *_t (\delta - \ell)$, $\hat \K(R,\cdot) := \hat\varrho^R$ and $\hat K(R,\cdot) := \hat\rho^R$ 
for $R\geq 0$, it follows that $\K = K *_t \M_R$. From this and $\frac{\partial \hat K}{\partial R} = -\beta\,\hat K$, 
we infer
\begin{equation*}
   \nabla \K = \frac{x}{|x|}\,\frac{\partial \K}{\partial R} \qquad\mbox{and}\qquad 
   \Delta \K = \frac{2}{|x|}\,\frac{\partial \K}{\partial R} + \frac{\partial^2 \K_\T}{\partial R^2}
\end{equation*}
with
\begin{equation*} 
   \frac{\partial \K}{\partial R} = -D_\beta(\K) + K *_t \frac{\partial \M_R}{\partial R} \qquad \mbox{and}
\end{equation*}
\begin{equation*}
   \frac{\partial^2 \K}{\partial R^2} 
     = D_\beta^2(\K) 
       - 2\,D_\beta(K) *_t \frac{\partial \M_R}{\partial R}
       + K *_t \frac{\partial^2 \M_R}{\partial R^2}\,.
\end{equation*}
We recall that $\hat \G = \hat G_0\,\hat \K$, $\hat G = \hat G_0\,\hat K$ and that the classical spherical 
wave $G_0$ satisfies 
\begin{equation*}
   \nabla G_0(x,\cdot) 
     = - \frac{x}{|x|}\,\left( \frac{1}{c_0}\frac{\partial G_0}{\partial t}(x,\cdot) + \frac{G_0(x,\cdot)}{|x|}\,\Id\right) 
\end{equation*}     
and
\begin{equation*}
   \Delta G_0(x,t) =  \frac{1}{c_0^2}\frac{\partial^2 G_0}{\partial t^2}(x,t) - \delta(t)\,\delta(x) 
\end{equation*}
on $\R^3\times \R$. Employing the above identities and $\delta(x)\,\K(x,t) = \delta(x)$ to 
$$
       \Delta \hat\G 
            = (\Delta\hat G_0)\,\hat \K 
              + 2\,(\nabla\hat G_0)\,(\nabla\hat \K) 
              + \hat G_0\,(\Delta \K) 
$$ 
yields equation~(\ref{contdisswaveeq}) with 
\begin{equation*}
\begin{aligned}
     - \L(x,\cdot) 
          := - 2\,\D_{c_0,\beta}(G) *_t \frac{\partial \M_R}{\partial R} 
             + G *_t \frac{\partial^2 \M_R}{\partial R^2}\,,
\end{aligned}
\end{equation*}
But this is nothing else but the claimed control term, if  
$$
     \frac{\partial \M_R}{\partial R} = - D_\mu(\M_R)    \quad\mbox{and}\quad 
     \frac{\partial^2 \M_R}{\partial R^2} =   D_\mu^2(\M_R)
     \quad\mbox{with}\quad \hat\mu := \log(1-\hat\ell)   
$$
are taken into account. As was to be shown. 
\end{proof}

\begin{coro}\label{coro:contprobdisswave}
Let $\T\in (0,\infty)$ and $\K(R,\cdot) := \F^{-1}\left( \K_1^R\right)$, where $\K_1$ satisfies 
$\supp(\K_1)=[0,\T]$ and 
$$
   \A(\K_1)(t) = \delta(t) - \ell(t)  \quad\mbox{on $\R$}\quad  \mbox{with}\quad  
   \K_1|_{t<0} = 0  \,,
$$ 
for an ordinary differential operator $\A$ with constant coefficients and a causal distribution $\ell$.   
Moreover, let $K(R,\cdot) := \F^{-1}\left( K_1^R\right)$, where $K_1$ solves the above equation with $\ell$ 
replaced by the zero function and let $\G$, $G$ and $\beta$ be defined as in Definition~\ref{defi:modeldiiswave}. 
Then $\G$ is the unique solution of~(\ref{contdisswaveeq}) with the control function $\L$ from Proposition~\ref{prop:contprobdisswave}. 
\end{coro}

\begin{proof}
The claim follows from the fact that 
$$
      \hat \K(R,\omega) 
           = (\hat \K(1,\omega))^R 
           = (1 - \hat\ell(\omega))^R\,\hat K(1,\omega)^R 
           = (1 - \hat\ell(\omega))^R\,\hat K(R,\omega) \,
$$
i.e. $\M_R = \F^{-1} \left( (1 - \hat\ell)^R \right)$,which is the same identity as in 
Proposition~\ref{prop:contprobdisswave}. 
\end{proof}

\begin{rema}
Let $\A(a,b,\ldots)$ be a partial operator with nonconstant coefficients $a,b,\ldots$ and 
e.g. $\A_0 := \A(a_0,b_0,\ldots)$ with $a_0:=a(0)$, $b_0:=b(0)$, $\ldots$. Moreover, let $\K$ and $K$ be as in Corollary~\ref{coro:contprobdisswave} with $\A$ replaced by $\A(a,b,\ldots)$. Then there exist a causal 
distribution $g(a,b,\ldots)$ such that $K$ and $\K$ are the unique solution of 
$$
   \A_0(K_1) = g(a,b,\ldots)  \quad\mbox{on $\R$}\quad  \mbox{with}\quad  K_1|_{t<0} = 0  \,
$$ 
and
$$
   \A_0(\K_1) = (\delta - \ell) *_t g(a,b,\ldots)  \quad\mbox{on $\R$}\quad  \mbox{with}\quad  \K_1|_{t<0} = 0  \,,
$$ 
respectively. Here we just reformulate the operator, e.g. 
$$
   a\,v' = \delta  \quad\Leftrightarrow\quad  a_0\,v' = \delta - (a-a_0)\,v' \,.
$$
But this implies $\K_1 = K_1 *_t g *_t (\delta - \ell)$ and 
$$
   \mbox{if $\hat g$ is nonnegative,} 
$$ 
then 
$\hat \K = \hat K *_t \hat g^R \, (1 - \hat\ell)^R$ for $R\geq 0$ is well-defined. In this case,  
Corollary~\ref{coro:contprobdisswave} remains true if $\M_R$ is replaced by 
$\F^{-1}\left( \hat g^R\,(1-\hat\ell)^R \right)$. 
\end{rema}

We conclude this section with a short discussion of the general control problem. Let $\T\in (0,\infty)$, 
$c_0\in (0,\infty)$ and $f$ be a causal tempered distribution with $\supp(f(x,\cdot)) = [0,\T_f(x)]$, where 
$\T_f(x) \in (0,\infty)$ for each $x\in\R^3$. We are interested in a wave $u$ (with front speed $c_0$) 
and a control functions $\L$ such that  
\begin{equation*}
\begin{aligned}
    &\left( D_\alpha + \frac{1}{c_0^2}\,\frac{\partial u}{\partial t} \right)^2\,u
               - \Delta u 
             = f - \L *_{x,t} f 
       \quad\mbox{on}\quad \R^3\times\R \quad\mbox{with}\,  \\
    &u|_{t<0} = 0  \,
\end{aligned}
\end{equation*}
and 
\begin{equation*}
    \mbox{ $\supp(u(x,\cdot))$ is bounded for each $x\in\R^3$\,.}
\end{equation*}
Then, as shown before, the simpler control problem with $f=\delta$ has a solution, say $\G$ for some control 
function $\L$. From Corollary~\ref{coro:contprobdisswave} and the Theorem of supports (cf. \cite{Ho03}), 
we infer that 
\begin{equation}
    u = \G *_{x,t} f \quad\mbox{with}\quad\supp(u(x,\cdot)) = [T_0(x),T_0(x) + T_f(|x|) + |x|\,\T) \,,
\end{equation}
i.e. this control problem is always solvable.

\section{Conclusion}

It is known that processes described by an ode do not stop within a finite time period and that processes 
described by a parabolic ode with konstant coefficients do not have a finite front speed, which is not always  
physically reasonable. (For a special analysis about the latter subject, we refer to~\cite{Ko11} and~\cite{Ko13}.) 
Moreover, as far as we know, all common wave equation models (pde with constant coefficients and common integro 
differential models for waves) consists in oscillations at various positions in space that do not stop within a 
finite time period. We consider this also not always physically reasonable. In this paper, we discussed and 
derived theorems for controlling odes and pdes such that their solutions correspond to 
\begin{itemize}
\item [(i)] relaxation and dissipative oscillation processes that stop at a finite time, 

\item [(ii)] diffusion processes that have a finite front speed and 

\item [(iii)] dissipative waves consisting of oscillations that stop within finite time periods.   

\end{itemize}
For the latter case, we considered only those dissipative waves which have a finite front speed. \\
In our research, we completely examined relaxations with finite stopping time and presented several useful 
examples, which can be used for the problems in (ii) and (iii) as well as stochastics. The case of dissipative 
oscillation is more complex and thus, due to limit of space, we only performed a basic analysis. But these 
results together with our results presented in~\cite{Ko19b} form a relatively thorough handling of dissipative 
oscillations as well as dissipative waves. It is obvious that controlling pdes to guratnee certain outcomes 
is very difficult and therefore we based our approach on results for controlling odes that are related to 
the considered pdes. In particular, we derived the relations beteween the control functions in the ode model 
and the respective pde model.

We hope that this work inspires other scientists that works in control theory, pde theory and inverse problem theory.

\section{Appendix: Two applications of the Paley-Wiener-Schwartz Theorem}

In this appendix, we present an application of the Paley-Wiener-Schwartz Theorem that will be used in our analysis 
of dissipative waves with frequency dependent attenuation laws.

For convenience, we start with a summary of some notation and Theorems about the Fourier transform. 
For more details, we refer to~\cite{GaWi99,Ho03,DaLi92_5}. 
We use the following form of the Fourier transform of an $L^1-$function $f$ 
$$
     \F(f)(\omega) := \hat f(\omega) := \int_\R f(t)\, e^{\i\,2\,\pi\,\omega\,t}\,\d t
     \qquad\mbox{for}\qquad \omega\in\R\,.
$$
Then the Convolution Theorem for $L^1-$functions reads as follows
$$
      \F(f *_t g) = \F(f)\,\F(g)
$$
and, if $f\in L^1(\R)$ is differentiable, then 
$$
         \F(f')(\omega) = (-\i\,2\,\pi\,\omega)\,\F(f)(\omega) \qquad\mbox{for}\qquad \omega\in\R\,.
$$
The inverse Fourier transform is denoted by $\check f$ and $\F^{-1}(f)$.

\begin{lemm}\label{lemm:fourier}
For fixed $\omega_0 > 0$, $a_0 \geq 0$ and $t\in\R$ let 
$$
   u(t) := \sin(\omega_0\,t)\,H(t)\,,\quad 
   v(t) := e^{-a_0\,t}\,\cos(\omega_0\,t)\,H(t) 
$$
and
$$
   w(t) := t\,\cos(\omega_0\,t)\,H(t)\,,
$$
where $H$ denotes the \emph{Heaviside function}. Then we have $\omega\in\R$ 
$$
  \hat u(\omega) = \frac{\omega_0}{\omega_0^2 + (-\i\,2\,\pi\,\omega)^2}  \,,\qquad\quad
  \hat v(\omega) = \frac{ a_0 + (-\i\,2\,\pi\,\omega)}{ (a_0 + (-\i\,2\,\pi\,\omega))^2 + \omega_0^2}  
$$
and
$$
  \hat w(t) 
    = \frac{1}{\omega_0^2 + (-\i\,2\,\pi\,\omega)^2} - \frac{2\,\omega_0^2}{[\omega_0^2 + (-\i\,2\,\pi\,\omega)^2]^2} \,.  
$$
\end{lemm}

\begin{proof}
Let $\delta$ denote the dirac distribution on $\R$. The first and the third claim follow from the fact that $u$ and $w$ 
solved the second order ode $g'' + \omega_0^2\,g = f$ on $\R$ with $g|_{t<0}=0$ and forcing term $f= \omega_0\,\delta$ 
and $f=\delta-2\,\omega_0\,u$, respectively. The second claim follows from a standard Fourier table, but it 
can also be derived in the same manner as the other claims. 
\end{proof}

In order to prove that a tempered distribution has support in $[0,\T]$ for some $\T\in(0,\infty]$, the following 
Theorem is essential.

\begin{theo}[Paley-Wiener-Schwartz] \label{theo:PWS}
$f\in \S'(\R)$ has support in $[0,\T]$ if and only if
\begin{itemize}
\item [(C1)] $z\in\C \mapsto \hat f(z)$ is entire and 

\item [(C2)] there exist constants $C$ and $N$ such that
$$
   |f(z)| \leq C\,(1+|z|)^N\,\exp\{ \sup_{t\in [0,\T]} (t\,\Im(z)) \}  \qquad \mbox{for}\qquad z\in\C\,.
$$
\end{itemize}
\end{theo}

The following two Propositions show that a relaxation function $\K_1$ (satisfying 
$\|\K_1\|_{L^1(\R)} \leq \frac{1}{\sqrt{2}}$) with infinite or finite stopping time 
is of the form $\F^{-1}(e^{-\alpha})$, where $\alpha$ is a (complex) \emph{attenuation law}.  

\begin{prop}\label{prop:alpha1}
Let $\K_1$ be a positive, monotonic decreasing causal function (, i.e. $\K_1|_{t<0}=0$) that is an element 
of $L^1(\R)$ and satisfies $\|\K_1\|_{L^1(\R)} \leq \frac{1}{\sqrt{2}}$. 
Then there exists a dissipation law $\alpha$ such that $\hat \K_1 := e^{-\alpha}$ holds 
and $\K(r,\cdot) :=  \F^{-1}\left( \hat\K_1^r\right)$ is well-defined for $r\geq 0$. 
\end{prop}

\begin{proof} 
Existence of $\alpha$. Because $\K_1\in L^1(\R)$, its Fourier transform exists and 
$$
    \Re(\hat\K_1)(\omega) 
       = \int_0^\infty \K_1\,\cos(2\,\omega\,t) \,\d t  
         \quad\mbox{and}\quad  
    \Im(\hat\K_1)(\omega) 
       = \int_0^\infty \K_1\,\sin(2\,\omega\,t) \,\d t \,.        
$$ 
Visualizing the graph of the integrand of the first integral shows that 
$\Re(\hat\K_1)(\omega) > 0$ for $\omega\in\R$ and thus $\Re(\hat\K_1)^{-1}$ and  
$a:=\log\left( \Re(\hat\K_1)^{-1} \right)$  exist and 
$$
    \hat\K_1
       = e^{-\log\left( \Re(\hat\K_1)^{-1} \right)}\,
            \left( 1 + \i\,\frac{\Im(\hat\K_1)}{\Re(\hat\K_1)} \right) \,.
$$
(To prove $\Re(\hat\K_1) > 0$ we required that $\K_1$ is positive and decreasing on $(0,\infty)$.)
Similarly as above, it follows that $\Im(\hat\K_1)(\omega) \geq 0$ for $\omega\in\R$ and therefore, 
for each $\omega\in\R$, there exists a unique $\varphi(\omega)\in \left(0,\frac{\pi}{2}\right)$ such that 
$$
   \varphi(\omega) = \arg \left(  1 + \i\,\frac{\Im(\hat\K_1)(\omega)}{\Re(\hat\K_1)(\omega)} \right) 
            \qquad\mbox{satisfying}\qquad 
   \varphi(\omega) = \arctan\left( \frac{\Im(\hat\K_1)(\omega)}{\Re(\hat\K_1)(\omega)} \right) \,.
$$
As a consequence, it follows that 
$$
    e^{-\alpha}
    \equiv  \hat\K_1
       = e^{ -\log\left( \Re(\hat\K_1)^{-1} \right) 
             + \log \sqrt{ 1 + \frac{\Im(\hat\K_1)^2}{\Re(\hat\K_1)^2} } 
             + \i\,\varphi} 
         \,
$$
or equivalently 
$$
  \alpha = \log\left( \Re(\hat\K_1)^{-1} \right) 
               + \frac{1}{2}\,\log \left( \frac{\Re(\hat\K_1)^2}
                                               {\Re(\hat\K_1)^2 + \Im(\hat\K_1)^2} \right)
               - \i\,(\varphi + 2\,\pi\,m) \,,
$$
where $m\in\Z$. Simplification and choosing $m=0$ yields 
\begin{equation}\label{repalphaT}
  \alpha = \frac{1}{2}\,\log \left( \frac{1}{ \Re(\hat\K_1)^2 + \Im(\hat\K_1)^2 } \right)
               - \i\, \arctan\left( \frac{\Im(\hat\K_1)}{\Re(\hat\K_1)} \right)\,.
\end{equation}
If $\Re(\hat\K_1)^2 + \Im(\hat\K_1)^2$ has values within $(0,1]$, 
then $\Re(\alpha)$ is non-negative and thus $\alpha$ is a dissipation law. This property is true, due to  
$$
    0   <   \Re(\hat\K_1)^2 + \Im(\hat\K_1)^2 
      \leq  ( |\Re(\hat\K_1)| + |\Im(\hat\K_1)| )^2 
      \leq  2\,\|\K_1\|_{L^1(\R)}^2
$$ 
and the assumption  
$\|\K_1\|_{L^1(\R)} \leq \frac{1}{\sqrt{2}}$. \\
Well-definedness of $\K(r,\cdot)$ for $r\geq 0$. Because the complex exponential function satisfies 
$(e^{z})^r=e^{z\,r}$ for $z\in\C$ and $r\in\R$, it follows that $\K(r,\cdot) = \F^{-1}(e^{-\alpha\,r})$ exists 
and is well-defined for $r\geq 0$. 
\end{proof}

\begin{rema}
Condition $\|\K_1\|_{L^1(\R)} \leq \frac{1}{\sqrt{2}}$ in Proposition~\ref{prop:alpha1} is not optimal. Indeed, 
it is easy to see that there exists an attenuation law $\alpha$ such that 
$$
     \K_1(t) := e^{-a_0\,t}\,H(t) = \F^{-1}(e^{-\alpha}) 
             \qquad\mbox{if and only if}\qquad  
     a_0^2 \geq 1\,.
$$
Of course, we are only interestes in positive constants $a_0$, i.e. $a_0>1$. According to 
Proposition~\ref{prop:alpha1}, we require $ \|\K_1\|_{L^1(\R)} = \frac{1}{a_0} \leq \frac{1}{\sqrt{2}}$, 
i.e. $a_0\geq \sqrt{2} \approx 1.414 >1$. Thus the general condition is not optimal. 
\end{rema}

It is clear that the previous theorem is not true if $\K_1$ is an arbitrary oscillation. However, 
there are examples of oscillations that can be represented by attenuation laws. Here we present one with 
infinite stopping time (cf. Proposition~\ref{prop:oscillrelax}).

\begin{exam}\label{exam:oscillrelax}
Let $a_0\geq 1$, $\omega_0\in (0,a_0)$, $b_0 := \omega_0^2 + a_0^2$ and  
$$
       K_1(t) := e^{-a_0\,t}\,\cos(\omega_0\,t)\,H(t)  \qquad\mbox{for}\qquad t\in\R\,.
$$ 
Then 
$$
  \hat K_1 (\omega) = \frac{ a_0 - \i\,2\,\pi\,\omega }{ \left(a_0 - \i\,2\,\pi\,\omega\right)^2 + \omega_0^2 } 
    \qquad\mbox{for}\qquad \omega\in\R
$$
and 
$$
  \alpha_1(\omega) := \log \frac{1}{|\hat K_1(\omega)|} > 0\quad\mbox{for}\quad \omega\in\R   
            \quad\Leftrightarrow\quad 
  a_0 \geq 1 \quad\mbox{and}\quad \omega_0 \in (0,a_0) \,.
$$
Thus there exists an attenuation law $\alpha$ with $\Re(\alpha(\omega)) = \alpha_1(\omega)$ such that 
$\hat K = e^{-\alpha}$. Moreover, $u$ is the unique solution of 
$$
         K_1''(t) + 2\,a_0\,K_1'(t) + b_0\,K_1(t) = \delta'(t) + a\,\delta(t) \quad\mbox{for}\quad t\in\R 
        \quad\mbox{with}\quad K_1|_{t<0}=0\,. 
$$
We note that $K(0+) = 1$ and $K'(0+) = -a$. 
\end{exam}

\begin{prop}\label{prop:alpha2}
Let $\K_1$, $\alpha$ and $\K(r,\cdot)$ ($R\geq 0$) be as in Proposition~\ref{prop:alpha1} with the additional 
assumption $\supp\left( \K_1\right) = [0,\T]$ for some $\T\in (0,\infty)$. If $\alpha$ is entire, then 
$\supp (\K(r,\cdot)) = [0,r\,T]$ for $r\geq 0$. 
\end{prop}

\begin{proof} 
Let $\beta_1$ and $\beta_2$ be such that $\alpha(z) = \beta_1(z) + \i\,\beta_2(z)$ with $\beta_1(z),\beta_2(z)\in\R$ 
for each $z\in\C$. Then we have $\left| e^{-\alpha(z)\,r} \right|  =  e^{ -\beta_1(z)\,R }$ with $r:=|\x|$. 
According to the Paley-Wiener-Schwartz Theorem (Theorem~\ref{theo:PWS} above), we have 
$\supp(\K_1) = [0,\T]$  if and only if 
(i) $z\in\C \mapsto \sqrt{2\,\pi}\,\hat \K_1(z)$ is entire and 
(ii) there exist constants $C$ and $N$ such that
$$
   e^{-\beta_1(z)} 
    \leq C\,(1+|z|)^N\,\exp\{ \sup_{t\in [0,\T]} (t\,\Im(z)) \}  \qquad \mbox{for}\qquad z\in\C\,.
$$
Because of our assumption, property (i) holds. The previous estimation for $e^{-\beta_1(z)}$ together with
\begin{equation}\label{helpPWtheo}
     e^{ -\beta_1(z)\,r } = \left( e^{ -\beta_1(z) } \right)^r 
          \qquad\mbox{and}\qquad
     r\,\sup_{t\in [0,\T]} (t\,\Im(z)) = \sup_{t\in [0,r\,\T]} (t\,\Im(z))
\end{equation}
for $r>0$ implies 
$$
   e^{-\beta_1(z)\,r} 
    \leq C\,(1+|z|)^{r\,N}\,\exp\{ \sup_{t\in [0,r\,\T]} (t\,\Im(z)) \}  \qquad \mbox{for}\qquad z\in\C\,,
$$
i.e. $\supp(\sqrt{2\,\pi}\,\K(r,\cdot)) = [0,r\,\T]$, which proves the claim. 
\end{proof}

\end{document}